\documentclass[reqno]{amsart}
\usepackage{amssymb,amscd}
\usepackage[T1,T2A]{fontenc}
\usepackage[utf8]{inputenc}
\usepackage[english,russian]{babel}
\usepackage[colorlinks,pagebackref,pdftitle={A first order prolongation of the conventional space},pdfdisplaydoctitle=true]{hyperref}

\let\sssize\scriptscriptstyle
\let\tsize\textstyle
\let\smc\scshape

\let\boldkey\boldsymbol
\newtheorem{proc}{Proposition}
\theoremstyle{definition}\newtheorem{definition}{Definition}
\theoremstyle{remark}\newtheorem{rem}{Remark}

\font\ssm=cmss10
\font\ssb=cmssbx10  
\font\sssm=cmss8
\def\p#1{\dfrac\partial{\partial#1}}
\def\pp#1#2{\dfrac{\partial#1}{\partial#2}}
\def\r{\overset{\sssize-1\,}r}
\def\x{\boldkey x}
\def\u{\boldkey u}
\def\du{{\bold{\dot{\boldkey u}}}}
\def\ddu{{\bold{\ddot{\boldkey u}}}}
\def\bd{{\boldsymbol\cdot}}
\def\z{{\boldkey z}}
\title[A first order prolongation of the conventional space]{A first order prolongation\\of the conventional space}
\author{R. Ya. Matsyuk}
\address{15 Dudayev St. 79005 L'viv Ukraine}
\email{romko.b.m@gmail.com; matsyuk@lms.lviv.ua}
\urladdr{{\url{http://www.iapmm.lviv.ua/12/eng/files/st_files/matsyuk.htm}}}
\begin{document}
\selectlanguage{english}
\bibliographystyle{amsplain}
{\vglue-12mm
\footnotesize
\leftline{\href{http://www.univie.ac.at/EMIS/proceedings/6ICDGA/IV/}
{\textbf{DIFFERENTIAL GEOMETRY AND APPLICATIONS}}}
\leftline{Proc. Conf., Aug.~28 -- Sept.~1, 1995, Brno, Czech Republic}
\leftline{Masaryk University, Brno 1996, 403--415}
\vskip32mm}
\begin{abstract}A variational equation of the third order in three-dimensional
space is proposed which describes autoparallel
curves of some connection
\end{abstract}
\maketitle
We shall focus on three-dimensional \hbox{(pseudo-)} \!Euclidean
space and consider the problem of finding a third-order variational-type
equation which can be put down in the form of the autoparallel transport equation
for some non-linear connection. It is common to introduce the latter
in the following form:
\begin{equation}
\dddot x^{\,\rho}=f^{\rho}(x^{\beta}, \dot x^{\beta}, \ddot x^{\beta})\,.
\label{1}
\end{equation}
On the other side, an Euler-Poisson
third-order equation is always of the affine type,
\begin{equation}
{\mathcal{A}}_{\rho\mu}\dddot x^{\,\mu}+k_{\rho}(x^{\beta}, \dot x^{\beta}, \ddot x^{\beta})=0\,,
\label{2}
\end{equation}
with a skew-symmetric matrix $\pmb{{\mathcal{A}}}=({\mathcal{A}}_{\rho\beta})$, and, consequently,
in the case when the number of equations equals {\it three\/}, can not be solved
with respect to the derivatives of the third order. What one can undertake in this
situation is at most to look for such a variational equation, which
describes the geodesic curve only up to reparametrization.
\section{General setting.}
One algorithm for building up an attached connection to a third order differential
equation of a certain class was presented in \cite{EvtushikKenerovo}, and we shall follow it here.
Although only \hbox{(pseudo-)} \!Euclidean space will be considered, to give the Reader a sense
of general setting, some constructions will be described as developed over
an $n$-dimensional manifold $M$. A differential equation of the third order
will be understood to be a cross-section of the third-order velocity
manifold
$T''M=J_{\sssize 3}({\mathbb{R}}_{\sssize 0};\, M)=\{\x;\u,\du,\ddu\}$,
fibred over the second order one,
$T'M=J_{\sssize 2}({\mathbb{R}}_{\sssize 0};\, M)=\{\x;\u,\du\}$.
These fibred manifolds are associated, as fibre bundles, to the principle
fibre bundles of the third-order and second-order frames,
$H''=\tilde J_{\sssize 3}({\mathbb{R}}^{n}_{\sssize 0};\, M)=\{x^{\rho};r^{\rho}_{\beta},r^{\rho}_{\beta\gamma},r^{\rho}_{\beta\gamma\delta}\}$,
and
$H'=\tilde J_{\sssize 2}({\mathbb{R}}^{n}_{\sssize 0};\, M)=\{x^{\rho};r^{\rho}_{\beta},r^{\rho}_{\beta\gamma}\}$,
where the tilde means that only invertible jets count, and also we shall denote
the inverse to the matrix $(r^{\rho}_{\beta})$ by $(\r^{\gamma}_{\delta})$.
The cotangent space to the manifold $H'$ is spanned by the following set
of differential forms (with coefficients from above the manifold $H''$)
\begin{align*}
\omega^{\rho}&\Doteq\r^{\rho}_{\mu}dx^{\mu}\,,\\
\omega^{\rho}_{\beta}&\Doteq\r^{\rho}_{\mu}dr^{\mu}_{\beta}
-\r^{\rho}_{\mu}r^{\mu}_{\beta\nu}\r^{\nu}_{\lambda}dx^{\lambda}\,,
\\
\begin{split}
\omega^{\rho}_{\beta\gamma}&\Doteq\r^{\rho}_{\mu}dr^{\mu}_{\beta\gamma}
-\r^{\rho}_{\mu}r^{\mu}_{\beta\nu}\r^{\nu}_{\lambda}dr^{\lambda}_{\gamma}
-\r^{\rho}_{\mu}r^{\mu}_{\gamma\nu}\r^{\nu}_{\lambda}dr^{\lambda}_{\beta}\,,
\\
&\qquad+\r^{\rho}_{\mu}r^{\mu}_{\beta\nu}\r^{\nu}_{\lambda}r^{\lambda}_{\gamma\iota}\r^{\iota}_{\sigma}dx^{\sigma}
+\r^{\rho}_{\mu}r^{\mu}_{\gamma\nu}\r^{\nu}_{\lambda}r^{\lambda}_{\beta\iota}\r^{\iota}_{\sigma}dx^{\sigma}
-\r^{\rho}_{\mu}r^{\mu}_{\nu\beta\gamma}\r^{\nu}_{\lambda}dx^{\lambda}\,.
\end{split}
\end{align*}
To span the cotangent space to the manifold $H''$ one more string of forms drops in
(we present their definition through a recurcive relation, which appears
more simple for any order as well),
$$
r^{\rho}_{\delta}\omega^{\delta}_{\iota\beta\gamma}=dr^{\rho}_{\iota\beta\gamma}
-r^{\rho}_{\mu\iota}\omega^{\mu}_{\beta\gamma}-r^{\rho}_{\mu\beta}\omega^{\mu}_{\iota\gamma}
-r^{\rho}_{\mu\gamma}\omega^{\mu}_{\beta\iota}-r^{\rho}_{\mu\iota\beta}\omega^{\mu}_{\gamma}
-r^{\rho}_{\mu\gamma\beta}\omega^{\mu}_{\iota}-r^{\rho}_{\mu\iota\gamma}\omega^{\mu}_{\beta}
-r^{\rho}_{\mu\iota\beta\gamma}\omega^{\mu}\,.
$$
These differential forms constitute a global object, intrinsically defined
in \cite{KobayashiCanonicalforms}.

Rather then proceed with the cross-section $f:\:T'M\to T''M$, one could wish to
develop some calculus on the corresponding principle bundles.
By the commutative diagram
\begin{equation}
\CD
H''\times{\mathbb{V}}''@>\rho''>>T''M\\
@A\Phi AA@AAfA\\
H'\times{\mathbb{V}}'@>>\rho'>T'M
\endCD
\label{3}
\end{equation}
the mapping $\Phi$ has to be both an equivariant one and a cross-section.
The typical fibre ${\mathbb{V}}''=J_{\sssize 3}({\mathbb{R}}_{\sssize 0};\, {\mathbb{R}}^{n}_{\sssize 0})
=\{U^{\rho},\dot U^{\rho},\ddot U^{\rho}\}$ undergoes such a left action
of the group
$GL''(n)=\tilde J_{\sssize 3}({\mathbb{R}}^{n}_{\sssize 0};\, {\mathbb{R}}^{n}_{\sssize 0})
=\{s^{\rho}_{\beta},s^{\rho}_{\beta\gamma},s^{\rho}_{\beta\gamma\delta}\}$,
that the quotient map $\rho''(r.s,s^{\sssize-1}.U)=\rho''(r,U)$ is described
explicitly by
\begin{align*}
u^{\rho}&=r^{\rho}_{\mu}U^{\mu}\,,\\
\dot u^{\rho}&=r^{\rho}_{\mu}\dot U^{\mu}+r^{\rho}_{\mu\nu}U^{\mu}U^{\nu}\,,\\
\ddot u^{\rho}&=r^{\rho}_{\mu}\ddot U^{\mu}+3r^{\rho}_{\mu\nu}\dot U^{\mu}U^{\nu}
+r^{\rho}_{\mu\nu\lambda}U^{\mu}U^{\nu}U^{\lambda}\,.
\end{align*}
A tangent vector to the product manifold $H''\times{\mathbb{V}}''$,
$$
\pmb{{\frak a}}=a^{\rho}\p{x^{\rho}}+a^{\rho}_{\mu}\p{r^{\rho}_{\mu}}
+a^{\rho}_{\mu\nu}\p{r^{\rho}_{\mu\nu}}
+a^{\rho}_{\mu\nu\lambda}\p{r^{\rho}_{\mu\nu\lambda}}
+\upsilon^{\rho}\p{U^{\rho}}+\dot\upsilon^{\rho}\p{\dot U^{\rho}}
+\ddot\upsilon^{\rho}\p{\ddot U^{\rho}}\,,
$$
is vertical with respect to the projection $\rho''$ if and only if
\begin{gather*}
a^{\rho}=0\,,
\\
r^{\rho}_{\mu}\upsilon^{\mu}+a^{\rho}_{\mu}U^{\mu}=0\,,\\
r^{\rho}_{\mu}\dot\upsilon^{\mu}+a^{\rho}_{\mu}\dot U^{\mu}
+a^{\rho}_{\mu\nu}U^{\mu}U^{\nu}-2r^{\rho}_{\mu\nu}\r^{\mu}_{\lambda}a^{\lambda}_{\iota}U^{\nu}U^{\iota}=0\,,\\
\begin{split}
r^{\rho}_{\mu}\ddot\upsilon^{\mu}+a^{\rho}_{\mu}\ddot U^{\mu}+3a^{\rho}_{\mu\nu}\dot U^{\mu}U^{\nu}
+a^{\rho}_{\mu\nu\lambda}U^{\mu}U^{\nu}U^{\lambda}
\qquad\qquad\qquad\qquad\\
-3r^{\rho}_{\mu\nu}\r^{\mu}_{\lambda}U^{\nu}(a^{\lambda}_{\iota}\dot U^{\iota}
	+a^{\lambda}_{\iota\sigma}U^{\iota}U^{\sigma}
	-2r^{\lambda}_{\iota\sigma}\r^{\iota}_{\beta}a^{\beta}_{\delta}U^{\sigma}U^{\delta})
\qquad&\\
\qquad\qquad
-3r^{\rho}_{\mu\nu}\r^{\nu}_{\lambda}a^{\lambda}_{\iota}\dot U^{\mu}U^{\iota}
-3r^{\rho}_{\mu\nu\lambda}\r^{\lambda}_{\iota}a^{\iota}_{\sigma}U^{\mu}U^{\nu}U^{\sigma}
&=0\,.
\end{split}
\end{gather*}

If the map $\Phi$ is equivariant, then its Lie derivative with respect to
an arbitrary pair of vertical vector fields $\pmb{{\frak a}}^{\boldsymbol\prime\boldsymbol\prime}$
and $\pmb{{\frak a}}^{\boldsymbol\prime}$ on the manifolds $H''\times{\mathbb{V}}''$ and $H'\times{\mathbb{V}}'$
is zero (i.e. vector fields $\pmb{{\frak a}}^{\boldsymbol\prime\boldsymbol\prime}$
and $\pmb{{\frak a}}^{\boldsymbol\prime}$ are $\Phi$-related):
$$
T\Phi\,\circ\,\pmb{{\frak a}}^{\boldsymbol\prime\boldsymbol\prime}\,=\,\pmb{{\frak a}}^{\boldsymbol\prime}\,\circ\,\Phi\,.
$$
The map $\eta^{\sssize-1}\Phi:\:H''\times{\mathbb{V}}'\to H''\times{\mathbb{V}}''$,
induced by the projection $\eta:\:H''\to H'$,
$$
\eta^{\sssize-1}\Phi(r,U)=\Phi(\eta r,U)\,,
$$
is not fibred over the identity in $H''$. Nevertheless, there can always be found
an element $(\delta^{\rho}_{\mu},0,s^{\rho}_{\mu\nu\lambda})\in GL''(n)$
such that
\begin{equation}
r^{\rho}_{\sigma}s^{\sigma}_{\mu\nu\lambda}+r^{\rho}_{\mu\nu\lambda}
=\Phi^{\rho}_{\mu\nu\lambda}(x^{\beta};r^{\beta}_{\iota},r^{\beta}_{\iota\gamma},U^{\beta},\dot U^{\beta})\,.
\label{4}
\end{equation}
We define the fibred morphism $F$ over the identity in $H''$ as a family of
cross-sections $F(r)$ of the fibration $\pi:\:{\mathbb{V}}''\to{\mathbb{V}}'$ by means of
$$
F^{\rho}(x^{\beta};r^{\beta}_{\mu},r^{\beta}_{\mu\nu},r^{\beta}_{\mu\nu\lambda},U^{\beta},\dot U^{\beta})
=\ddot\Phi^{\rho}+s^{\rho}_{\iota\sigma\gamma}U^{\iota}U^{\sigma}U^{\gamma}\,,
$$
with $s^{\rho}_{\iota\sigma\gamma}$ defined from~(\ref{4}).
Then, by virtue of $\Phi$ equivariant, for every vertical vector field
$\pmb{{\frak a}}$ on the manifold $H''\times{\mathbb{V}}''$ we have
\begin{equation}
TF\,\circ\,(\,T(id)\times T\pi)\,\circ\,\pmb{{\frak a}}\,
=\,
\pmb{{\frak a}}
\,\circ\,F\,\circ\,(id\times\pi)\,,
\label{5}
\end{equation}
as can be seen from the next diagram by an appropriate `diagram chasing',
$$
\CD
TH''\times T{\mathbb{V}}''@<TF<<TH''\times T{\mathbb{V}}'\\
@| @|\\
TH''\times T{\mathbb{V}}''@>T(id)\times T\pi>>TH''\times T{\mathbb{V}}'\\
@A\pmb{{\frak a}}AA @AA(T(id)\times T\pi)\,\circ\,\pmb{{\frak a}}A\\
H''\times {\mathbb{V}}''@>id\times \pi>>H''\times{\mathbb{V}}'\\
@| @|\\
H''\times {\mathbb{V}}''@<F<<H''\times{\mathbb{V}}'
\endCD
$$

The kernel of $T\rho''$ annuls the following one-forms:
\begin{equation}
\begin{alignedat}{2}
&\omega^{\rho}\,,\\
&\triangle U^{\rho}&&\Doteq dU^{\rho}+U^{\mu}\omega^{\rho}_{\mu}\,,\\
&\triangle \dot U^{\rho}&&\Doteq d\dot U^{\rho}+\dot U^{\mu}\omega^{\rho}_{\mu}
+U^{\mu}U^{\nu}\omega^{\rho}_{\mu\nu}\,,\\
&\triangle \ddot U^{\rho}&&\Doteq d\ddot U^{\rho}+\ddot U^{\mu}\omega^{\rho}_{\mu}
+3U^{\mu}\dot U^{\nu}\omega^{\rho}_{\mu\nu}+U^{\mu}U^{\nu}U^{\lambda}\omega^{\rho}_{\mu\nu\lambda}\,.
\end{alignedat}
\label{6}
\end{equation}
By calculating the Lie derivative of the differential forms~(\ref{6}) along a vertical
vector field it turns out that the exterior differential system, generated by~(\ref{6}) is invariant under the action of the group $GL''(n)$ upon the manifold
$H''\times{\mathbb{V}}''$.

{\it If the functions $F^{\rho}$ satisfy~(\ref{5}), then the differential
forms
$$
\triangle F^{\rho}\Doteq dF^{\rho}+\ddot U^{\mu}\omega^{\rho}_{\mu}
+3U^{\mu}\dot U^{\nu}\omega^{\rho}_{\mu\nu}+U^{\mu}U^{\nu}U^{\lambda}\omega^{\rho}_{\mu\nu\lambda}
$$
expand into the differential forms~(\ref{6}) alone:}
\begin{equation}
\triangle F^{\rho}
=F^{\rho}{}^{\sssize2}_{\mu}\triangle\dot U^{\mu}
+F^{\rho}{}^{\sssize1}_{\mu}\triangle U^{\mu}
+F^{\rho}{}^{\sssize0}_{\mu}\omega^{\mu}\,.
\label{7}
\end{equation}

The concept of second order connection involves the quotient manifold
$\Gamma H'=TH'/GL'(n)$ with respect to the standard action of the group $GL'(n)$:
$$
a^{\rho}\p{x^{\rho}}+a^{\rho}_{\mu}\p{r^{\rho}_{\mu}}
+a^{\rho}_{\mu\nu}\p{r^{\rho}_{\mu\nu}}\mapsto
a^{\rho}\p{x^{\rho}}+s^{\nu}_{\mu}a^{\rho}_{\nu}\p{r^{\rho}_{\mu}}
+(s^{\lambda}_{\mu\nu}a^{\rho}_{\lambda}
+s^{\lambda}_{\mu}s^{\iota}_{\nu}a^{\rho}_{\lambda\iota})\p{r^{\rho}_{\mu\nu}}\,.
$$
The manifold $\Gamma H'$ projects onto the manifold $T''M$ by means of the following
mapping of $TH'$, compatible with this action:
$$
(a^{\rho},a^{\rho}_{\mu},a^{\rho}_{\mu\nu})\mapsto
(a^{\rho},a^{\rho}_{\mu}\r^{\mu}_{\nu}a^{\nu},
a^{\rho}_{\lambda\iota}\r^{\lambda}_{\mu}\r^{\iota}_{\nu}a^{\mu}a^{\nu}
\!\!-\!a^{\rho}_{\delta}\r^{\delta}_{\sigma}r^{\sigma}_{\lambda\iota}\r^{\lambda}_{\mu}\r^{\iota}_{\nu}a^{\mu}a^{\nu}
\!\!+\!a^{\rho}_{\iota}\r^{\iota}_{\lambda}a^{\lambda}_{\mu}\r^{\mu}_{\nu}a^{\nu})\,.
$$
\begin{definition}[\cite{EvtushikTretyakov_translation}]
A second order connection is given by a map $\gamma:\:T'M\to\Gamma H'$,
which is identity in $TM$.
\end{definition}
By means of the commutative diagram
$$
\CD
\Gamma H'@>>>T''M\\
@A\gamma AA@AA\widehat\gamma A\\
T'M@=T'M
\endCD
$$
every connection $\gamma$ defines a morphism of manifolds $\widehat\gamma$.
The very similar way to~(\ref{7}) this map $\gamma$
may be described through the structure equations,
\begin{alignat*}{2}
&\triangle\Gamma^{\rho}_{\beta}
&&=\Gamma^{\rho}_{\beta}{}^{\sssize2}_{\mu}\triangle\dot U^{\mu}
+\Gamma^{\rho}_{\beta}{}^{\sssize1}_{\mu}\triangle U^{\mu}
+\Gamma^{\rho}_{\beta}{}^{\sssize0}_{\mu}\omega^{\mu}\,,\\
&\triangle\Gamma^{\rho}_{\beta\gamma}
&&=\Gamma^{\rho}_{\beta\gamma}{}^{\sssize2}_{\mu}\triangle\dot U^{\mu}
+\Gamma^{\rho}_{\beta\gamma}{}^{\sssize1}_{\mu}\triangle U^{\mu}
+\Gamma^{\rho}_{\beta\gamma}{}^{\sssize0}_{\mu}\omega^{\mu}\,,
\end{alignat*}
where the differential forms $\triangle\Gamma^{\rho}_{\beta}$, $\triangle\Gamma^{\rho}_{\beta\gamma}$
are build up from the differentials of the functions ($\Gamma^{\rho}_{\beta},\, \Gamma^{\rho}_{\beta\gamma}$),
which represent the map $\gamma$, as follows:
\begin{alignat*}{2}
&\triangle\Gamma^{\rho}_{\beta}&&\Doteq d\Gamma^{\rho}_{\beta}
+\Gamma^{\mu}_{\beta}\omega^{\rho}_{\mu}
-\Gamma^{\rho}_{\mu}\omega^{\mu}_{\beta}
+U^{\mu}\omega^{\rho}_{\beta\mu}\,,\\
&\triangle\Gamma^{\rho}_{\beta\gamma}&&\Doteq d\Gamma^{\rho}_{\beta\gamma}
+\Gamma^{\mu}_{\beta\gamma}\omega^{\rho}_{\mu}
-\Gamma^{\rho}_{\beta\mu}\omega^{\mu}_{\gamma}
-\Gamma^{\rho}_{\mu\gamma}\omega^{\mu}_{\beta}
+\Gamma^{\mu}_{\beta}\omega^{\rho}_{\mu\gamma}
+\Gamma^{\mu}_{\gamma}\omega^{\rho}_{\mu\beta}
-\Gamma^{\rho}_{\mu}\omega^{\mu}_{\beta\gamma}
+U^{\mu}\omega^{\rho}_{\beta\gamma\mu}\,.
\end{alignat*}

Connection $\gamma$ is called {\it stable\/} if it projects onto the identity
in $T'M$. In this case the morphism $\widehat\gamma$ is a cross-section and
thus defines a third-order differential equation of type~(\ref{1}).

To discuss a weaker condition of a quasi-stable connection, we recall that
the group $GL'(1)$ acts on the right upon the space $T'M$ by parameter transformations.
The generators are:
\begin{align*}
\pmb{{\frak p}}_{\sssize\bold1}&=u^{\rho}\p{u^{\rho}}+2\dot u^{\rho}\p{\dot u^{\rho}}\,,\\
\pmb{{\frak p}}_{\sssize\bold2}&=u^{\rho}\p{\dot u^{\rho}}\,.
\end{align*}
The quotient space with respect to this action is the manifold $C'M$ of contact
elements, locally arranged as ${\mathbb{R}}\times T'{\mathbb{R}}^{n-1}$. Connection $\gamma$
is said to be {\it quasi-stable\/} if it projects onto the identity in $C'M$.

In case of quasi-stable connection it is possible to introduce\cite{EvtushikKenerovo,EvtushikTretyakov_translation} the
notion of parallel transport in such a way, that the autoparallel curves of this
connection will be described in the typical fibre ${\mathbb{V}}''$ of the fibre bundle $T''M$
by means of the equation
\begin{equation}
\ddot U^{\rho}=\Gamma^{\rho}_{\mu}\dot U^{\mu}+\Gamma^{\rho}_{\mu\nu}U^{\mu}U^{\nu}
+\lambda^{\sssize(2)}\dot U^{\rho}+\lambda^{\sssize(1)}U^{\rho}\,.
\label{9}
\end{equation}
If the quasi-stable connection $\gamma$ is stable, the
functions $\lambda^{\sssize(1)}$ and $\lambda^{\sssize(2)}$ both vanish.

Not every equation~(\ref{1}) can be rearranged in the form~(\ref{9}).
The crucial idea consists in applying a somewhat technical trick of reparametrization.
If the map $f$ in~(\ref{1}) or~(\ref{3}) defines in the consistent manner
some equation on the manifold $C'M$, and if we think of $f$ as of a vector field $\pmb{{\frak f}}$
on the manifold $T'M$ by the inclusion $T''M\hookrightarrow TT'M$, then the Lie brackets
$
[\pmb{{\frak p}}_{\sssize1},\pmb{{\frak f}}]$ and $[\pmb{{\frak p}}_{\sssize2},\pmb{{\frak f}}]
$
differ from a multiply of $\pmb{{\frak f}}$ by some vertical field with respect to the projection
\begin{equation}
\wp:\:T'M\to C'M\,.
\label{10}
\end{equation}
In fact, a stronger condition holds:
$$
\left\{
\begin{aligned}
(T\wp)\,[\pmb{{\frak p}}_{\sssize\bold1},\pmb{{\frak f}}]&=(T\wp)\,\pmb{{\frak f}}\\
(T\wp)\,[\pmb{{\frak p}}_{\sssize\bold2},\pmb{{\frak f}}]&=0
\end{aligned}
\right.
$$
In terms of the representation~(\ref{7}) the above condition amounts
to the following two equations with Lagrange multiplies $\mu$ and $\lambda$,
\begin{align}
3F^{\rho}-F^{\rho}{}^{\sssize1}_{\iota}U^{\iota}
-2F^{\rho}{}^{\sssize2}_{\iota}\dot U^{\iota}&=3\mu U^{\rho}
\label{11}
\\
3\dot U^{\rho}-F^{\rho}{}^{\sssize2}_{\iota}U^{\iota}&=3\lambda U^{\rho}\,.
\label{12}
\end{align}
The multipliers $\mu$ and $\lambda$ are functions on the manifold $H'\times{\mathbb{V}}'$,
and in order them to represent some well-defined functions on the manifold
$T'M$, they both have to satisfy the condition of $GL'(n)$-invariance of the type
$
\langle\pmb{{\frak a}},d\mu\rangle=0
$
for any $\rho'$-vertical vector $\pmb{{\frak a}}$, which amounts to the following system of partial
differential equations:
\begin{equation}
\begin{gathered}
\pp{\mu}{\dot U^{\rho}}U^{\beta}U^{\gamma}=r^{\nu}_{\rho}\pp{\mu}{r^{\nu}_{\beta\gamma}}\\
\pp{\mu}{U^{\rho}}U^{\beta}-\pp{\mu}{\dot U^{\rho}}\dot U^{\beta}
=2r^{\delta}_{\rho\nu}\pp{\mu}{r^{\delta}_{\beta\nu}}+r^{\delta}_{\rho}\pp{\mu}{r^{\delta}_{\beta}}\,.
\end{gathered}
\label{13}
\end{equation}
\begin{definition}[\cite{EvtushikTretyakov_translation}]
The equation~(\ref{1}) is {\it reducible\/} if (\ref{11}, \ref{12}) holds for the
representation~(\ref{7}). It will be called {\it strictly reducible\/}
if both $\mu=0$ and $\lambda=0$.
\end{definition}

Consider now a (second order nonlinear) connection,
the coefficients  $\Gamma^{\rho}_{\beta}$, $\Gamma^{\rho}_{\beta\gamma}$
of which are constructed from the coefficients of the first-order prolongation
of the differential system~(\ref{7}),
$$
dF^{\rho}{}^{\sssize2}_{\beta}+F^{\mu}{}^{\sssize2}_{\beta}\omega^{\rho}_{\mu}
-F^{\rho}{}^{\sssize2}_{\mu}\omega^{\mu}_{\beta}+3U^{\mu}\omega^{\rho}_{\beta\mu}
=F^{\rho}{}^{\sssize2}_{\beta}{}^{\sssize0}_{\mu}\omega^{\mu}+F^{\rho}{}^{\sssize2}_{\beta}{}^{\sssize1}_{\mu}
\triangle U^{\mu}+F^{\rho}{}^{\sssize2}_{\beta}{}^{\sssize2}_{\mu}\triangle \dot U^{\mu}\,,
$$
according to the following prescription:
\begin{equation}
\begin{alignedat}{2}
&\Gamma ^{\rho}_{\beta}&&={\tsize\frac13}F^{\rho}{}^{\sssize2}_{\beta}\,;\\
&\Gamma ^{\rho}{}_{\beta\gamma}&&={\tsize\frac12}(\Pi^{\rho}_{\beta\gamma}+\Pi^{\rho}_{\gamma\beta})\,,
\quad\text{where} \\
&\Pi^{\rho}_{\beta\gamma}&&={\tsize\frac13}F^{\rho}{}^{\sssize2}_{\beta}{}^{\sssize0}_{\gamma}
+{\tsize\frac19}(F^{\rho}{}^{\sssize2}_{\beta}{}^{\sssize2}_{\mu}F^{\mu}{}^{\sssize1}_{\gamma}
+F^{\rho}{}^{\sssize2}_{\beta}{}^{\sssize1}_{\mu}F^{\mu}{}^{\sssize2}_{\gamma})
+{\tsize\frac2{27}}F^{\rho}{}^{\sssize2}_{\beta}{}^{\sssize2}_{\mu}F^{\mu}{}^{\sssize2}_{\nu}F^{\nu}{}^{\sssize2}_{\gamma}\,.
\end{alignedat}
\label{14}
\end{equation}

Let us agree to call the connection, constructed accordingly to the formulae~(\ref{14}), as one, {\it attached\/} to the differential equation~(\ref{1})
\begin{proc}[\cite{EvtushikKenerovo}]
The connection, attached to a reducible differential equation, is quasi-stable.
The equation~(\ref{9}) of the autoparallel curves  of the connection,
attached to a reducible differential equation,
coincides with the initial equation~(\ref{1}).
If~(\ref{1}) is strictly reducible, then the attached connection is stable.
\end{proc}
The functions $\lambda^{\sssize(1)}$ and $\lambda^{\sssize(2)}$ in~(\ref{9})
are expressed through the functions $\mu$ from~(\ref{11}) and
$\lambda$ from~(\ref{12}) in terms
of the coefficients of the differential $d\lambda$,
$$
d\lambda=\lambda^{\sssize0}_{\mu}\omega^{\mu}+\lambda^{\sssize1}_{\mu}\triangle U^{\mu}
+\lambda^{\sssize2}_{\mu}\triangle\dot U^{\mu}\,,
$$
according to the formulae below:
\begin{gather*}
\lambda^{\sssize(1)}=\lambda^{\sssize0}_{\nu}U^{\nu}+\lambda^{\sssize1}_{\nu}\dot U^{\nu}
\lambda^{\sssize2}_{\nu}\ddot U^{\nu}+\mu(1-\lambda^{\sssize2}_{\nu}U^{\nu})
-\lambda(\lambda^{\sssize1}_{\nu}U^{\nu}
+{\tsize\frac23}\lambda^{\sssize2}_{\nu}F^{\nu}{}^{\sssize2}_{\iota}U^{\iota})-2\lambda^{2}\,,
\\\lambda^{\sssize(2)}=2\lambda\,.
\end{gather*}

In view of~(\ref{13}),
$$
\lambda^{\sssize1}_{\rho}=\pp{\lambda}{U^{\rho}}\,,
\qquad\lambda^{\sssize2}_{\rho}=\pp{\lambda}{\dot U^{\rho}}\,.
$$
\section{Euclidean space. Variational equation.}
As declared, we look for a third order differential equation in \hbox{(pseudo-)} \!Euclidean
space ${\bold E}^{\sssize 3}$, which would be derivable from a Lagrangian.
The dimension of the space is {\it three}. As mentioned at the very beginning
of the present contribution, we cannot expect such equation to exist in the form,
solved with respect to the highest (i.e of the third order) derivatives.
So we shall first settle down on the manifold
$$
{\mathbb{R}}\times T'{\bold E}^{\sssize 2}
$$
and afterwards go all the way back to the manifold $T'{\bold E}^{\sssize 3}$ along the
projection of~(\ref{10}), which in the canonical coordinates is so expressed:
\begin{align*}
\text{\ssm t}\circ\wp&=x^{\sssize0}\\
\text{\ssm x}^{a}\circ\wp&=x^{a}\\
\text{\ssm v}^{a}\circ\wp&=\frac{u^{a}}{u^{\sssize0}} \\
\text{\ssm v}'{}^{a}\circ\wp&=\frac{\dot u^{a}}{u_{\sssize0}{}^{2}}
	-\frac{\dot u_{\sssize0}}{u_{\sssize0}{}^{3}}u^{a}\,.
\end{align*}

Let us concentrate on a system of two third-order ordinary differential equations
\begin{equation}
{\text{\ssm E}}_{a}=0\,.
\label{15}
\end{equation}
We introduce a vector valued differential one-form
\begin{equation}
\boldsymbol\epsilon={\text{\ssm E}}_{a}\,{{d}}{\text{\ssm x}}^{a}\otimes{{d}}\text{\ssm t}\,,
\label{16}
\end{equation}
where the expressions ${\text{\ssm E}}_{a}$ are called the Euler-Poisson
expressions.\footnote{This is an alternative way to interpret the notion of
the Euler morphism, the latter having been considered by {\smc Kol\'a\v r}
in \cite{KolarNoveMesto}.}
Applying the general criterion of \cite{TulczyjewResolution} for an arbitrary system
of differential equations to be a system of Euler-Poisson equations, it was
established in \cite{Theses} that the vector expression ${\text{\ssb E}}=\{{\text{\ssm E}}_{a}\}$
in~(\ref{15}) must have the shape
\begin{equation}
{\text{\ssb E}}={\mathbb{A}}{\,\boldkey.\,}{\text{\ssb v}}^{\boldsymbol\prime\boldsymbol\prime}{\,+\,}({\text{\ssb v}}^{\boldsymbol\prime}{\!\boldkey.\,}{\boldsymbol\partial}_{\pmb{\text{\sssm v}}})\,
{\mathbb{A}}{\,\boldkey.\,}{\text{\ssb v}}^{\boldsymbol\prime}{\,+\,}{\mathbb{B}}{\,\boldkey.\,}{\text{\ssb v}}^{\boldsymbol\prime}{\,+\,}{\text{\ssb c}}\,,
\label{17}
\end{equation}
where  the skew-symmetric matrix ${\mathbb{A}}$, the matrix ${\mathbb{B}}$, and the column vector
${\text{\ssb c}}$ depend on the variables $\text{\ssm t}$, ${\text{\ssb x}}$, ${\text{\ssb v}}={d{\text{\ssb x}}}/d\text{\ssm t}$,
and satisfy the following system of partial differential equations in $\text{\ssm t}$, ${\text{\ssm x}}^a$,
and ${\text{\ssm v}}^a$ \cite{Theses,SovPhys30_6}
\begin{equation}
\begin{gathered}
       \partial_{_{_{_{{\text{\sssm v}}}}}}{\!}_{[a}{}{\text{\ssm A}}_{bc]}=0 \\[1\jot]
       2\,{\text{\ssm B}}_{[ab]}-3\,{\bold D}_{\sssize{\bold 1}}{\,}{\text{\ssm A}}_{ab}=0\\[1\jot]
       2\,\partial_{_{_{_{{\text{\sssm v}}}}}}{\!}_{[a}{}{\text{\ssm B}}_{b]c}
-4\,\partial_{_{_{_{{\text{\sssm x}}}}}}{\!}_{[a}{}{\text{\ssm A}}_{b]c}
+{\partial_{_{_{_{{\text{\sssm x}}}}}}{\!}_{c}}{\,}{\text{\ssm A}}_{ab}
+2\,{\bold D}_{\sssize{\bold 1}}{\,}{\partial_{_{_{_{{\text{\sssm v}}}}}}{\!}_{c}}{\,}{\text{\ssm A}}_{ab}=0 \\[1\jot]
       {\partial_{_{_{_{{\text{\sssm v}}}}}}{\!}_{(a}}{}{\text{\ssm c}}_{b)}
-{\bold D}_{\sssize{\bold 1}}{\,}{\text{\ssm B}}_{(ab)}=0\\[1\jot]
       2\,{\partial_{_{_{_{{\text{\sssm v}}}}}}{\!}_{c}}{\,}\partial_{_{_{_{{\text{\sssm v}}}}}}{\!}_{[a}{}{\text{\ssm c}}_{b]}
-4\,\partial_{_{_{_{{\text{\sssm x}}}}}}{\!}_{[a}{}{\text{\ssm B}}_{b]c}
+{{\bold D}_{\sssize{\bold 1}}}^{2}{\,}{\partial_{_{_{_{{\text{\sssm v}}}}}}{\!}_{c}}{\,}{\text{\ssm A}}_{ab}
+6\,{\bold D}_{\sssize{\bold 1}}{\,\,}\partial_{_{_{_{{\text{\sssm x}}}}}}{\!}_{[a}{}{\text{\ssm A}}_{bc]}=0 \\[1\jot]
       4\,\partial_{_{_{_{{\text{\sssm x}}}}}}{\!}_{[a}{}{\text{\ssm c}}_{b]}
-2\,{\bold D}_{\sssize{\bold 1}}{\,\,}\partial_{_{_{_{{\text{\sssm v}}}}}}{\!}_{[a}{}{\text{\ssm c}}_{b]}
-{{\bold D}_{\sssize{\bold 1}}}^{3}{\,}{\text{\ssm A}}_{ab}=0\,.
\end{gathered}
\label{18}
\end{equation}
In~(\ref{18}) ${\bold {}D_{\sssize1}}$ and farther below ${\bold {}D_{\sssize2}}$ denote the generators of the Cartan distribution,
\begin{align*}
{\bold {}D_{\sssize2}}={\text{\ssb v}}^{\boldsymbol\prime}{\!\boldkey.\,}{\boldsymbol\partial}_{\pmb{\text{\sssm v}}}
{\,+\,}&{\bold {}D_{\sssize1}}\,,\\
&{\bold {}D_{\sssize1}}=\partial_{\text{\sssm t}}{\,+\,}{\text{\ssb v}}{\,\boldkey.\,}{\boldsymbol\partial}_{\pmb{\text{\sssm x}}}\,.
\end{align*}

Let $\boldsymbol\theta_{\sssize{\bold 2}}$, $\boldsymbol\theta_{\sssize{\bold 3}}$ denote the canonical contact
forms
\begin{align*}
\boldsymbol\theta_{\sssize{\bold 3}}
=\frac\partial{\partial {{\text{\ssm v}}'}^{a}}\otimes({{d}}{\text{\ssm v}'}^{a}
-{\text{\ssm v}''}^{a}{{d}}\text{\ssm t}){\,+\,}&\boldsymbol\theta_{\sssize{\bold 2}}\,,  \\
&\boldsymbol\theta_{\sssize{\bold 2}}
=\frac\partial{\partial {\text{\ssm v}}^{a}}\otimes({{d}}\text{\ssm v}^{a}
-{\text{\ssm v}'}^{a}{{d}}\text{\ssm t})
{\,+\,}\frac\partial{\partial {\text{\ssm x}}^{a}}\otimes({{d}}\text{\ssm x}^{a}
-\text{\ssm v}^{a}{{d}}\text{\ssm t}) \,.
\end{align*}

Along with the differential form $\boldsymbol\epsilon$
we introduce another one, $\pmb{\underline{\epsilon}}$,
\begin{align*}
\pmb{\underline{\epsilon}}
=\text{\ssm A}_{ab}\,{{d}}\text{\ssm x}^{a}\otimes{{d}}{\text{\ssm v}'}^{b}
{\,+\,}&\text{\ssm k}_{a}\,{{d}}\text{\ssm x}^{a}\otimes{{d}}\text{\ssm t}\,,\\
&\text{\ssb k}
=({\text{\ssb v}}^{\boldsymbol\prime}{\!\boldkey.\,}{\boldsymbol\partial}_{\pmb{\text{\sssm v}}})\,
{\mathbb{A}}{\,\boldkey.\,}{\text{\ssb v}}^{\boldsymbol\prime}{\,+\,}{\mathbb{B}}{\,\boldkey.\,}{\text{\ssb v}}^{\boldsymbol\prime}{\,+\,}{\text{\ssb c}}\,.
\end{align*}
Exterior differential systems, generated by the forms $\boldsymbol\epsilon$ and
$\pmb{\underline{\epsilon}}$, are equivalent:
$$
\pmb{\underline{\epsilon}}-\boldsymbol\epsilon
=(\text{\ssm A}_{ab}\,{{d}}\text{\ssm x}^{a}\otimes{{d}}{\text{\ssm v}'}^{b})
\barwedge\boldsymbol\theta_{\sssize{\bold 3}}\,.
$$

Now it is time to put in the concept of {\it symmetry}. Let
\begin{equation}
\pmb{{\frak x}}=\tau\frac\partial{\partial\text{\ssm t}}{\,+\,}{{\frak x}^{a}}\frac\partial{\partial\text{\ssm x}^{a}}
\label{19}
\end{equation}
denote the generator of some local group of transformations of the manifold ${\mathbb{R}}\times{{\bold E}}^{\sssize2}$,
its successive prolongations to the space
$J_{s}({\mathbb{R}};{{\bold E}}^{\sssize2})\approx{\mathbb{R}}\times T^{s}{{\bold E}}^{\sssize2}$
denoted by $\pmb{{\frak x}_{s}}$:
$$\pmb{{\frak x}_{\sssize 2}}={\frak v}^{a}\frac\partial{\partial{\text{\ssm v}'}^{a}}
{\,+\,}\pmb{{\frak x}_{\sssize 1}}\,.
$$

The demand that the exterior differential system, generated by the vector
valued differential form $\pmb{\underline{\epsilon}}$, be invariant under the
infinitesimal transformation $\pmb{{\frak x}}$ incarnates into the following
equation\footnote{The notion of vector bundle valued exterior differential systems
invariance was introduced in \cite{Methods}, see also \cite{SymVectorForms}.}
\begin{equation}
{\bold L}(\pmb{{\frak x}_{\sssize 2}})(\pmb{\underline{\epsilon}})=\boldsymbol\Xi{\,\boldkey.\,}
\pmb{\underline{\epsilon}}{\,+\,}\boldsymbol\beta\barwedge\boldsymbol\theta_{\sssize{\bold 2}}\,,
\label{20}
\end{equation}
where the elements of the matrix
$\boldsymbol\Xi$
and the coefficients of the semi-basic $T^{\ast}{{\bold E}}^{\sssize2}$-valued one-form
$\boldsymbol\beta$
depend upon the variables $\text{\ssm t}$, $\text{\ssb x}$, $\text{\ssb v}$, and ${\text{\ssb v}}^{\boldsymbol\prime}$.
Both $\boldsymbol\Xi$ and $\boldsymbol\beta$ play the role of Lagrange multipliers.
Splitting of equation~(\ref{20})
with respect to independent differentials
${{d}}\text{\ssm t}$, ${{d}}\text{\ssm x}^{a}$, ${{d}}\text{\ssm v}^{a}$, and
${{d}}{\text{\ssm v}'}^{a}$, results in the following system of partial differential
equations
\begin{equation}
\begin{aligned}
{\bold L}(\pmb{{\frak x}_{\sssize 1}})\text{\ssm A}_{ab}&={\Xi_{a}}^{c}\,\text{\ssm A}_{cb}
-\text{\ssm A}_{ac}\frac\partial{\partial{\text{\ssm v}'}^{b}}{\frak v}^{c}\\[1\jot]
{\bold L}(\pmb{{\frak x}_{\sssize 2}})\text{\ssm k}_{a}&={\Xi_{a}}^{b}\,\text{\ssm k}_{b}
-\text{\ssm A}_{ab}{\bold {}D_{\sssize2}}{\frak v}^{b}
-\text{\ssm k}_{a}{\bold {}D_{\sssize1}}\tau\,.
\end{aligned}
\label{21}
\end{equation}

\subsection{Variational problem in parametric form.}
Consider for a moment an $r^{\text th}$-order variational problem in parametric form, set by a Lagrangian
$$
\ell(\zeta, x^{\rho}, u^{\rho}, \dots,{\overset{\sssize r-1}u}{}^{\rho}){{d}}\zeta
$$
on the space $J_{r}({\mathbb{R}};M)$. As long as we limit ourselves only
to the case of autonomous Euler-Poisson equations,
\begin{equation}
{\mathcal{E}}_{\rho}=0\,,\;\text{see~(\ref{2})},
\label{22}
\end{equation}
the differential form
\begin{equation}
\boldsymbol\varepsilon={\mathcal{E}}_{\rho}\,{{d}}x^{\rho}\otimes{{d}}\zeta
\label{23}
\end{equation}
may {\it globally\/} be deprived of the factor ${{d}}\zeta$, constituting
thus a {\it globally\/} defined $T^{\ast}M$-valued density
\begin{equation}
\boldkey e={\mathcal{E}}_{\rho}{{d}}x^{\rho}\,.
\label{24}
\end{equation}

Now the projection $\wp:T^{r}M\to C^{r}M$ can be employed to
generate an autonomous variational problem set over $T^{r}M$ from every one
variational problem over $C^{r}M$.

\begin{proc}
In terms of a local chart, if in~(\ref{16}) the local semi-basic differential
form $\boldsymbol\epsilon$ corresponds to the Lagrangian
$$
L{{d}}\text{\ssm t}\,,
$$
then the vector valued density
\begin{equation}
\boldkey e=-u^{a}(\text{\ssm E}_{a}\circ\wp){{d}}x^{\sssize0}
{\,+\,}u^{\sssize0}(\text{\ssm E}_{a}\circ\wp){{d}}x^{a}
\label{25}
\end{equation}
corresponds to the Lagrangian
$$
\ell(\zeta, x^{\rho}, u^{\rho}, \dots,{\overset{\sssize r-1}u}{}^{\rho})\,{{d}}\zeta
={\mathcal{L}}(x^{\rho}, u^{\rho}, \dots,{\overset{\sssize r-1}u}{}^{\rho})\,{{d}}\zeta
$$
with the Lagrange function
\begin{equation}
{\mathcal{L}}=u^{\sssize0}L{\,\circ\,}\wp\,.
\label{26}
\end{equation}
\end{proc}

Let us return to the third-order case. The relations between quantities,
allocated on the space of contact elements $C'M\overset{\text{def}}{=}C^{2}M$ and the corresponding
quantities on the second-order velocity space $T'M\overset{\text{def}}{=}T^{2}M$,
expressed by (\ref{22}, \ref{24}, and~\ref{25}), say, that in~(\ref{2}) we have
$$
\boldkey k\;=\;(\,\bold{\dot{\boldkey u}}\,\boldkey.\,\boldsymbol\partial_{\boldkey u}\,)\;\pmb{{\mathcal{A}}}\,\boldkey.\,\bold{\dot{\boldkey u}}
\;+\;\pmb{{\mathcal{B}}}\,\boldkey.\,\bold{\dot{\boldkey u}}\;+\;\boldkey c
$$
with
\begin{equation}
{\mathcal{A}}_{ab}=(u^{\sssize0})^{-2}\cdot(\text{\ssm A}_{ab}\circ\wp),
\quad{\mathcal{B}}_{ab}=(u^{\sssize0})^{-1}(\text{\ssm B}_{ab}\circ\wp),\quad c_{a}=u^{\sssize0}(\text{\ssm c}_{a}\circ\wp)\,,
\label{27}
\end{equation}
and that the Weierstrass constraint holds:
$$
\pmb{{\mathcal{A}}}\,\boldkey.\,\boldkey u\equiv0,\quad\boldkey k\,\boldkey.\,\boldkey u\equiv0\,.
$$
\subsection{Circles and hyperbolae.}
Let in~(\ref{19}) generator $\pmb{{\frak x}}$ correspond to the \hbox{(pseudo-)} \!ortho\-gonal
transformations of a three-dimensional \hbox{(pseudo-)} \!Euclidean plain. Solving~(\ref{18}) together with~(\ref{21}),
we establish the expressions~(\ref{17}) for this case (see~\cite{CMPh1998} for more details):
\begin{equation}
{\text{\ssb E}}
=-\dfrac{\ast\text{\ssb v}^{\boldsymbol\prime\boldsymbol\prime}}
{(1+\text{\ssb v}\boldsymbol\cdot\text{\ssb v})^{3/2}}
\;+\;3\,\dfrac{\ast\text{\ssb v}^{\boldsymbol\prime}}
{(1+\text{\ssb v}\boldsymbol\cdot\text{\ssb v})^{5/2}}
\boldkey(\text{\ssb v}^{\boldsymbol\prime}\!\boldsymbol\cdot\text{\ssb v}\boldkey)
\;+\;\dfrac m{(1+\text{\ssb v}\boldsymbol\cdot\text{\ssb v})^{3/2}}
\bigl[(1+\text{\ssb v}\boldsymbol\cdot\text{\ssb v})\text{\ssb v}^{\boldsymbol\prime}
\,-\,\boldkey(\text{\ssb v}^{\boldsymbol\prime}\!\boldsymbol\cdot\text{\ssb v}\boldkey)\text{\ssb v}\bigr]\,.
\label{28}
\end{equation}
(The dual to some vector $\text{\ssb w}$ is known to be defined with the help of the
skew-symmetric Levi-Civita symbol $\text{\ssm e}_{ab}$ by means of
$(\ast\text{\ssb w})_{a}=\text{\ssm e}_{ba}\text{\ssm w}^{b}$.)
To convert~(\ref{28}) into a ``homogeneous'' three-dimensional form one applies~(\ref{27}) and obtains the final Euler-Poisson equations, which are
naturally connected to the second prolongation of the transformation group
$[{\bold E}(3,i),{{\bold E}}^{\sssize3}]$:
\begin{equation}
\boxed{
\pmb{{\mathcal{E}}}
=\dfrac{\bold{\ddot{\boldkey u}}\times\boldkey u}{\|\boldkey u\|^{3}}
\;-\;3\,\dfrac{\bold{\dot{\boldkey u}}\times\boldkey u}{\|\boldkey u\|^{5}}
\boldkey(\bold{\dot{\boldkey u}}\boldsymbol\cdot\boldkey u\boldkey)
\;+\;m\,\dfrac {\bold{\dot{\boldkey u}}\,\boldkey(\boldkey u\boldsymbol\cdot\boldkey u\boldkey)
\,-\,\boldkey u\,\boldkey(\bold{\dot{\boldkey u}}\boldsymbol\cdot\boldkey u\boldkey)}
{\|\boldkey u\|^{3}}=0
}
\label{29}
\end{equation}
Furthermore, we can indicate a general formula for the family of the Lagrange functions
which produce the expression~(\ref{29}):
$$
{\mathcal{L}}_{(\rho)}
\;=\;\dfrac{u^{\rho}[\bold{\dot{\boldkey u}},\boldkey u,\boldkey e_{(\rho)}]}
{\|\boldkey u\|\,\|\boldkey u\times\boldkey e_{(\rho)}\|^{2}}
\;-\;m\,\|\boldkey u\|
\;+\;\bold{\dot{\boldkey u}}\,\boldkey.\,\boldsymbol\partial_{\boldkey u}\;\phi
\;+\;\boldkey a\,\boldkey.\,\boldkey u\,,
$$
where an arbitrary row vector $\boldkey a$ is constant and a function $\phi$ depending on the variable
$\boldkey u$ is subject to the constraint $\boldkey u\,\boldkey.\,\boldsymbol\partial_{\boldkey u}\,\phi\;=\;0$.
(Recall the notation $[\;,\;,\;]$ for the parallelepipedal product of three vectors.)
The vector $\boldkey e_{\rho}$ denotes the $\rho$-th component of the
\hbox{(pseudo-)} \!Euclidean frame.
{\it Each ${\mathcal{L}}_{(\rho)}$ fits in.}

Although {\it there does not exist an invariant (even in extended sense)
Lagrange  function}, the equations~(\ref{29}) {\it are invariant\/} with respect to the group under consideration.
Namely, let
\begin{equation}
\pmb{{\frak x}''}\;=\;\boldkey[\boldsymbol\varpi, \boldkey x,
\boldsymbol\partial_{\boldkey x}\boldkey]
\;+\;\boldkey[\boldsymbol\varpi, \boldkey u, \boldsymbol\partial_{\boldkey u}\boldkey]
\;+\;\boldkey[\boldsymbol\varpi, \bold{\dot{\boldkey u}}, \boldsymbol\partial_{\bold{\dot{\boldkey u}}}\boldkey]
\;+\;\boldkey[\boldsymbol\varpi, \bold{\ddot{\boldkey u}}, \boldsymbol\partial_{\bold{\ddot{\boldkey u}}}\boldkey]
\label{30}
\end{equation}
stand for the third-order prolongation of the infinitesimal \hbox{(pseudo-)} \!Euclidean
transformations to the manifold $T''M$ with $\boldsymbol\varpi$ for the
group parameter. Then
$$
{\bold L}(\pmb{{\frak x}''})(\pmb{{\mathcal{E}}})\;=\;\boldsymbol\varpi\times\pmb{{\mathcal{E}}}\,.
$$
\begin{rem}
Assuming $m=0$ in~(\ref{29}), we recover geodesic circles as
integral curves, and in the case the index $i$ in ${\bold E}(3,i)$ equals $2$
this amounts to uniformly accelerated motion in three-dimensional special relativity.
\end{rem}
\section{Euclidean space. Connection.}
In a \hbox{(pseudo-)} \!orthonormal frame of reference, the corresponding third-order
frame takes on the shape
\begin{equation}
r^{\rho}_{\beta}=\delta^{\rho}_{\beta},\quad r^{\rho}_{\beta\gamma}=0,\quad
r^{\rho}_{\beta\gamma\nu}=0\,,
\label{31}
\end{equation}
so the structure forms $\omega^{\rho}_{\beta}$, $\omega^{\rho}_{\beta\gamma}$,
and $\omega^{\rho}_{\beta\gamma\nu}$ vanish and we can identify
$\u$, $\du$, and $\ddu$ with the ``invariant coordinates'' $U$, $\dot U$, and $\ddot U$ respectively.

In order to construct a connection, consistent with the equation~(\ref{29}),
we first supplement the two independent expressions, entering in~(\ref{29}), with an arbitrary additional one. Without loss of
generality we can search for the latter in the form
\begin{equation}
\ddu\,\boldkey.\,\u\;=\;\|\u\|^{2}\cdot\Psi(\u,\du)\,.
\label{32}
\end{equation}
The system of equations~(\ref{29} and~\ref{32}) can now be solved with
respect to the third-order derivatives to produce:
\begin{equation}
\ddu\;=\;3\,\dfrac{\du\bd \u}{\|\u\|^{2}}\,\du\;-\;3\,\dfrac{(\du\bd \u)^{2}}{\|\u\|^{4}}\u
\;-\;m\,\u\times\du\;+\;\Psi\cdot\u\,.
\label{33}
\end{equation}
Now we proceed further to define more precisely the arbitrary function $\Psi$.
With (\ref{11} and~\ref{12}) we calculate $\mu$ and $\lambda$ for the equation~(\ref{33}):
\begin{align}
\mu\;&=\;\dfrac1{3}\,\left(\,2\Psi\,-\,2\du\pp{\Psi}{\du}\,-\,\u\pp{\Psi}{\u}\,\right)\,,
\label{34}
\\
\lambda\;&=\;\dfrac{\u\bd\du}{\|\u\|^{2}}\,-\,\dfrac1{3}\,\u\pp{\Psi}{\du}\,.
\label{35}
\end{align}
In the reference frame~(\ref{31}) by virtue of~(\ref{13}) we conclude that
$\mu$ and $\lambda$ are constant. Then the compatibility conditions for the system
of partial differential equations~(\ref{34}) and~(\ref{35}) show, that $\lambda$
must be equal to zero,
\begin{equation}
\lambda=0\,.
\label{36}
\end{equation}
Set
$$
\Psi\;=\;\dfrac3{\|\u\|^{2}}\,\left(\,\psi\;+\;\dfrac1{2}\|\du\|^{2}\right)\,.
$$
For the function $\psi$ we now get
\begin{equation}
\left(\u\bd\p{\du}\right)\,\psi\;=\;0\,.
\label{37}
\end{equation}
Introducing the intermediate variable $\z=\u\times\du$ we see by~(\ref{37})
that the function $\psi$ depends on $\u$ and $\du$ via the variable $\z$ only.
Now we express the equation~(\ref{34}) in terms of $\z$ to get
\begin{equation}
3\,\z\bd\pp{\psi}{\z}\;=\;4\psi\;-\;\|\u\|^{2}\mu\,.
\label{38}
\end{equation}
Again, the compatibility conditions for~(\ref{38}) turn $\mu$ to zero,
\begin{equation}
\mu=0\,.
\label{39}
\end{equation}

To settle the matter definitely, we call upon the demand of \hbox{(pseudo-)} \!Euclidean
symmetry (with the generator~(\ref{30})) for the equation~(\ref{33}), which gives
\begin{equation}
\z\times\p{\z}\;\psi\;=\;0\,.
\label{40}
\end{equation}
Altogether (\ref{38}, \ref{39}, and~\ref{40}) produce for the determination of the function $\psi$
the equation
$$
\pp{\psi}{\z}\;=\;=\dfrac4{3}\,\dfrac{\z}{\|\z\|^{2}}\,\psi
$$
with the solution $\psi=\|\z\|^{4/3}$. Thus the function $\Psi$ has been found,
$$
\Psi\;=\;\dfrac3{2}\,\dfrac{\|\du\|^{2}}{\|\u\|^{2}}\;+\;3A\,\dfrac{\|\du\times\u\|^{4/3}}{\|\u\|^{2}},\qquad
\text{$A$ is a real number}.
$$

Let us introduce the following obvious definition:
\begin{definition}
For any smooth transformation group $G$ acting upon a manifold $M$ let a curve
$\sigma:\:I\to M, \;I\subset{\mathbb{R}}$ be called an {\it autogeodesic path\/} if
\begin{enumerate}
\item the image $\sigma(I)$ is an extremal submanifold of some parameter-independent
variation problem;
\item the curve $\sigma$  is autoparallel with respect to some (nonlinear, higher-order)
connection on $M$;
\item the corresponding autoparallel transport equation is $G$-invariant.
\end{enumerate}
\end{definition}

In view of the preceding considerations we now are capable of calculating the coefficients
$(\Gamma^{\rho}_{\beta},\Gamma^{\rho}_{\beta\gamma})$ of the connection, given by~(\ref{14}).
Rather then make this, it appears more economic to accomplish only with the presentation
of the explicit expression for the corresponding autogeodesic path
equation.
\begin{proc}
\noindent\begin{enumerate}
\item
The third-order
autogeodesic paths of the three-dimensional \hbox{(pseudo-)} \!Eucli\-dean
space are the solutions of the next differential equation:
{
\begin{equation}
\boxed{
\ddu=3\dfrac{\du\bd \u}{\|\u\|^{2}}\du-3\left[\dfrac{(\du\bd \u)^{2}}{\|\u\|^{4}}
-\dfrac{\|\du\|^{2}}{2\|\u\|^{2}}-A\dfrac{\|\du\times\u\|^{4/3}}{\|\u\|^{2}}\right]\u
-m\,\u\times\du
}
\label{41}
\end{equation}
}
\item The corresponding connection is stable.
\end{enumerate}
\end{proc}
\begin{proof}
It is necessary to calculate $\Gamma^{\rho}_{\beta}$ and $\Gamma^{\rho}_{\beta\gamma}$ from
(\ref{14}) and to show that the right-hand side of~(\ref{9}) coincides with
the right-hand side of~(\ref{41}). The second statement follows from~(\ref{36} and~\ref{39}). \end{proof}


\providecommand{\bysame}{\leavevmode\hbox to3em{\hrulefill}\thinspace}
\providecommand{\MR}{\relax\ifhmode\unskip\space\fi MR }
\providecommand{\MRhref}[2]{%
  \href{http://www.ams.org/mathscinet-getitem?mr=#1}{#2}
}
\providecommand{\href}[2]{#2}

\enddocument